\documentclass[12pt,twoside]{amsart}
\usepackage{amssymb,amscd}
\usepackage{mathrsfs}
\usepackage{array,float}

\theoremstyle{plain}
\newtheorem{thm}{Theorem}[section]
\newtheorem{lem}[thm]{Lemma}
\newtheorem{pro}[thm]{Proposition}
\newtheorem{cor}[thm]{Corollary}
\newtheorem{conjecture}[thm]{Conjecture}
\newtheorem*{conjecture-prime}{Conjecture}
\newtheorem{question}[thm]{Question}

\theoremstyle{remark}
\newtheorem*{rem*}{Remark}

\theoremstyle{definition} 
\newtheorem{definition}[thm]{Definition}
\newtheorem{example}[thm]{Example}

\numberwithin{equation}{section}

\newcommand{\Z}{\mathbb{Z}}   
\newcommand{\Q}{\mathbb{Q}}
\newcommand{\N}{\mathbb{N}}
\newcommand{\F}{\mathbb{F}}

\DeclareMathOperator{\Cen}{C}
\DeclareMathOperator{\Zen}{Z}
\DeclareMathOperator{\rank}{rk}
\DeclareMathOperator{\GL}{GL}
\DeclareMathOperator{\sat}{sat}

\newcommand{\ackn}{  \noindent{\sc Acknowledgement }\hspace{5pt} }

\renewcommand{\phi}{\varphi}

\begin{document}

\title[Characterisation of uniform pro-$p$ groups]{A characterisation
  of uniform pro-$p$ groups}

\author{Benjamin Klopsch}
\address{Department of Mathematics \\
  Royal Holloway, University of London \\
  Egham TW20 0EX, UK}
\curraddr{Institut f\"ur Algebra und Geometrie, Mathematische 
  Fakult\"at, Otto-von-Guericke-Universit\"at Magdeburg, 39016
  Magdeburg, Germany}

\email{Benjamin.Klopsch@rhul.ac.uk}

\author{Ilir Snopce}
\address{Universidade Federal do Rio de Janeiro\\
  Instituto de Matem\'atica \\
  20785-050 Rio de Janeiro, RJ, Brasil }
\email{ilir@im.ufrj.br}

\begin{abstract}
  Let $p$ be a prime.  Uniform pro-$p$ groups play a central role in
  the theory of $p$-adic Lie groups.  Indeed, a topological group
  admits the structure of a $p$-adic Lie group if and only if it
  contains an open pro-$p$ subgroup which is uniform.  Furthermore,
  uniform pro-$p$ groups naturally correspond to powerful $\Z_p$-Lie
  lattices and thus constitute a cornerstone of $p$-adic Lie theory.

  In the present paper we propose and supply evidence for the
  following conjecture, aimed at characterising uniform pro-$p$
  groups.  Suppose that $p \geq 3$ and let $G$ be a torsion-free
  pro-$p$ group of finite rank.  Then $G$ is uniform if and only if
  its minimal number of generators is equal to the dimension of $G$ as
  a $p$-adic manifold, i.e., $d(G) = \dim(G)$.  In particular, we
  prove that the assertion is true whenever $G$ is soluble or $p >
  \dim(G)$.
\end{abstract}


\maketitle

\section{Introduction}

Throughout let $p$ be a prime.  Lazard's seminal paper \emph{Groupes
  analytiques $p$\nobreakdash-adiques}~\cite{La65}, published in 1965,
provides a comprehensive treatment of the theory of $p$-adic analytic
Lie groups.  One of his main results was a solution of the
$p$\nobreakdash-adic analogue of Hilbert's 5th problem.  More
precisely, he obtained the following algebraic characterisation of
$p$-adic analytic groups: a topological group is $p$-adic analytic if
and only if it contains a finitely generated open pro-$p$ subgroup
which is saturable.  In the 1980s Lubotzky and Mann introduced the
concept of a powerful pro-$p$ group and used this notion to
re-interpret the group-theoretic aspects of Lazard's work, by and
large sidestepping the analytic side of the theory.  Central to their
approach are the uniformly powerful pro-$p$ groups, which play a role
similar to the one of saturable pro-$p$ groups in Lazard's work.  A
detailed treatment of $p$-adic analytic groups from this point of view
and a sample of its manifold applications are given
in~\cite{DidSMaSe99}.

A pro-$p$ group $G$ is said to be \emph{powerful} if $p \geq 3$ and
$[G,G]\leq G^p$, or $p=2$ and $[G,G]\leq G^4$.  Here, $[G,G]$ and
$G^p$ denote the (closures of the) commutator subgroup and the
subgroup generated by all $p$th powers.  The definition of a uniformly
powerful pro-$p$ group is stated in Section~\ref{sec:uniform}.  For
the moment it suffices to recall that a pro-$p$ group is
\emph{uniformly powerful}, or \emph{uniform} for short, if and only if
it is finitely generated, powerful and torsion-free.  The \emph{rank}
of a pro-$p$ group $G$ is the basic invariant
\begin{displaymath}
  \rank(G) := \sup \{ d(H) \mid H \text{ a closed subgroup of } G  \},
\end{displaymath}
where $d(H)$ denotes the minimal cardinality of a topological
generating set for~$H$.  

The class of $p$-adic analytic groups can now be characterised as
follows; see~\cite[Corollary~8.34]{DidSMaSe99}.  A topological group
admits the structure of a $p$-adic analytic Lie group if and only if
it contains an open pro-$p$ subgroup of finite rank.  Moreover, a
pro-$p$ group has finite rank if and only if it admits a uniform open
subgroup.

A key invariant of a $p$-adic Lie group $G$ is its \emph{dimension} as
a $p$-adic manifold which we denote by~$\dim(G)$.  Algebraically,
$\dim(G)$ can be described as $d(U)$, where $U$ is any uniform open
pro-$p$ subgroup of~$G$.  In~\cite{Kl11} it is shown that, for $p \geq
3$, every torsion-free compact $p$-adic Lie group $G$ satisfies
$\rank(G) = \dim(G)$.

\subsection{Main results} \label{sec:main-results} The purpose of the
present paper is to propose a new characterisation of uniform pro-$p$
groups in terms of their minimal numbers of generators.  It is well
known that every uniform pro-$p$ group $G$ satisfies $d(G) = \dim(G)$.
We propose and supply evidence for the following conjecture.

\begin{conjecture} \label{con:infinite} Suppose that $p \geq 3$ and
  let $G$ be a torsion-free pro-$p$ group of finite rank.  Then $G$ is
  uniform if and only if $d(G) = \dim(G)$.
\end{conjecture}

Standard examples show that the assertion of the conjecture cannot
extend to $p=2$ without modifications; see Section~\ref{sec:uniform}.
It is not quite clear what one could reasonably hope for in this case
and for now we shall concentrate on $p \geq 3$.  Our first results
imply that Conjecture~\ref{con:infinite} is indeed true for soluble
pro-$p$ groups.  Every profinite group $G$ of finite rank has a
maximal finite normal subgroup.  We refer to this subgroup as the
\emph{periodic radical} of $G$ and denote it by~$\pi(G)$.

\begin{thm} \label{thm:soluble} Suppose that $p\geq 5$ and let $G$ be
  a soluble pro-$p$ group of finite rank such that $\pi(G)=1$.  If
  $d(G)=\dim(G)$ then $G$ is uniform.
\end{thm}

In Section~\ref{sec:uniform} we give an example to show that the
assertion of Theorem~\ref{thm:soluble} does not extend to $p=3$
without modifications.  Nevertheless a separate analysis leads to the
following somewhat weaker theorem, still confirming
Conjecture~\ref{con:infinite}.

\begin{thm} \label{thm:soluble-p-equals-3} Let $G$ be a torsion-free
  soluble pro-$3$ group of finite rank.  If $d(G)=\dim(G)$ then $G$ is
  uniform.
\end{thm}

For a finitely generated nilpotent group $\Gamma$ we denote by
$\Pi(\Gamma)$ the set of all primes $q$ such that $\Gamma$ contains an
element of order~$q$.  It is known that the torsion elements of
$\Gamma$ form a finite subgroup, hence the set $\Pi(\Gamma)$ is
finite.  As a consequence of the above theorems we obtain the
following corollary.

\begin{cor} \label{cor:nilpotent} Let $\Gamma$ be a finitely generated
  nilpotent group and let $h(\Gamma)$ denote the Hirsch length
  of~$\Gamma$.  Suppose that $p \notin \Pi(\Gamma) \cup \{2 \}$. Then
  the pro-$p$ completion $\widehat{\Gamma}_p$ of $\Gamma$ is a uniform
  pro-$p$ group if and only if $\dim_{\F_p}(\Gamma / \Gamma^p
  [\Gamma,\Gamma]) = h(\Gamma)$.
\end{cor}

Our current results for insoluble groups are not quite strong enough
to settle Conjecture~\ref{con:infinite} in full generality.  However,
we are able to confirm the conjecture for a wide range of groups.

From~\cite{FeGoJa08} we recall that a subgroup $N$ of a pro-$p$ group
$G$ is \emph{PF-embedded} in $G$ if there exists a central descending
series of closed subgroups $N_i$, $i \in \N$, starting at $N_1 = N$
such that $\bigcap_{i \in \N} N_i = 1$ and
$[N_i,\underbrace{G,\ldots,G}_{p-1}] \leq N_{i+1}^p$ for all $i \in
\N$.

\begin{thm} \label{thm:injective} Suppose that $p\geq 3$. Let $G$ be a
  pro-$p$ group of finite rank with an open PF-embedded subgroup and
  such that the map $G \to G$, $x \mapsto x^p$ is injective.  If $d(G)
  = \dim(G)$ then $G$ is uniform.
\end{thm}

It is not hard to see that every closed subgroup of a finitely
generated saturable pro-$p$ group has an open PF-embedded
subgroup. Since it is one of the built-in features of a saturable
pro-$p$ group $G$ that the map $G \to G$, $x \mapsto x^p$ is
injective, we obtain immediately the following corollary.

\begin{cor} \label{cor:sat-un} Suppose that $p\geq 3$ and let $G$ be a
  closed subgroup of a finitely generated saturable pro-$p$ group. If
  $d(G) = \dim(G)$ then $G$ is uniform.
\end{cor}

We recall that every uniform pro-$p$ group is saturable, albeit the
converse is not true; see~\cite{Kl05}.  The next corollary can be used
to produce easily examples of saturable groups which are not
uniform.

\begin{cor} \label{cor:char-sat-uniform} Suppose that $p \geq 3$ and
  let $G$ be a finitely generated saturable pro-$p$ group. Then $G$ is
  uniform if and only if $d(G)=\dim(G)$.
\end{cor}

Furthermore, a result of Gonz\'alez-S\'anchez and Klopsch
in~\cite{GoKl09} allows us to derive the following consequence.

\begin{cor} \label{cor:large-p} Let $G$ be a torsion-free pro-$p$
  group of finite rank such that $p > \dim(G)$.  If $d(G) = \dim(G)$
  then $G$ is uniform.
\end{cor}

Finally, it is natural to ask whether there is an analogue of
Conjecture~\ref{con:infinite} for finite $p$-groups.  Currently, there
is perhaps not enough evidence to support a `formal' conjecture, but
we can raise the following problem.  For a finite $p$-group $G$ the
subgroup $\Omega_1(G)$ is the group generated by all elements of order
$p$ in~$G$.

\begin{question} \label{con:finite} Suppose that $p \geq 3$ and let
  $G$ be a finite $p$-group.  Is it true that $G$ is powerful if and
  only if $d(G) = \log_p \lvert \Omega_1(G) \rvert$?
\end{question}

Similarly as for Conjecture~\ref{con:infinite}, the forward
implication in Question~\ref{con:finite} is known to be true.  Suppose
that $p \geq 3$ and let $G$ be a powerful finite $p$-group.  Then
$\lvert G : G^p \rvert = \lvert \Omega_1(G) \rvert$, by
\cite[Theorem~3.1]{Wi02}, and from $[G,G] \subseteq G^p$ we conclude
that $d(G) = \lvert G:G^p \rvert = \lvert \Omega_1(G) \rvert$.  The
actual problem is whether the implication in the other direction is
true.  Our next result shows that a positive answer to this question
would resolve Conjecture~\ref{con:infinite}.

\begin{pro} \label{pro:qu-implies-conj} Conjecture~\ref{con:infinite}
  is true if the answer to Question~\ref{con:finite} is `yes'.
\end{pro}

Section~\ref{sec:uniform} contains a variation of
Question~\ref{con:finite} and a short discussion, showing that the
answer is indeed `yes' for certain special classes of groups, namely
for finite $p$-groups which are regular, potent or $p$-central.

\subsection{Some applications} \label{sec:applications} The property
of being powerful is readily inherited by factor groups and by direct
products, but straightforward examples show that often it is not
inherited by subgroups.  We say that a pro-$p$ group $G$ is
\emph{hereditarily powerful} if every open subgroup of $G$ is
powerful.  Similarly, we say that $G$ is \emph{hereditarily uniform}
if every open subgroup of $G$ is uniform.  In \cite{LuMa87I} Lubotzky
and Mann proved that a finite $p$-group is hereditarily powerful if
and only if it is modular and, if $p=2$, not Hamiltonian.  From their
result we deduce the following.

\begin{thm}[Lubotzky and Mann] \label{thm:her-powerful} Let $G$ be a
  finitely generated pro-$p$ group.  Then $G$ is hereditarily powerful
  if and only if there exist an abelian normal subgroup $A$ of $G$, an
  element $b \in G$ and $s \in \N \cup \{ \infty \}$, with $s \geq 2$
  if $p=2$, such that
  \begin{displaymath}
    G = \langle b \rangle A,   
  \end{displaymath}
  where the group $A$ is written additively and $b$ acts as
  multiplication by $1+p^s$.
\end{thm}

\begin{cor} \label{cor:hered-uniform} Let $G$ be a finitely generated
  pro-$p$ group. Then $G$ is hereditarily uniform if and only if one
  of the following holds:
  \begin{enumerate}
  \item $G \cong \Z_p^d$ is abelian for some $d \in \{0 \} \cup \N$;
  \item $G \cong \langle b \rangle \ltimes A$ for $\langle b \rangle
    \cong \Z_p$ and $A \cong \Z_p^{d-1}$, where $d \geq 2$ and $b$
    acts on $A$ as multiplication by $1+p^s$ for some $s \in \N$, with
    $s \geq 2$, if $p=2$.
  \end{enumerate}
\end{cor}

In~\cite{KlSn11} the authors of the present paper used Lie ring
methods to classify all finitely generated pro-$p$ groups with
constant generating number on open subgroups; see Theorem~\ref{thm:Kl-Sn}.
In Section~\ref{sec:hereditarily} we indicate how
Corollary~\ref{cor:hered-uniform} yields an alternative proof in the
case $p \geq 3$, which does not require Lie ring techniques.
  
\subsection{Notation}
Throughout the paper, $p$ denotes a prime. The $p$-adic integers and
$p$-adic numbers are denoted by $\Z_p$ and $\Q_p$.  We write $C_p$ to
refer to a cyclic group of order~$p$.

Subgroups $H$ of a topological group $G$ are tacitly taken to be
closed and by generators we mean topological generators as
appropriate.  The minimal number of generators of a group $G$ is
denoted by $d(G)$.  Likewise the minimal cardinality of a generating
set for a module $M$ over a ring $R$ is denoted by $d_R(M)$. 

A pro-$p$ group $G$ is said to be \emph{just-infinite} if it is
infinite and every non-trivial normal subgroup of $G$ has finite index
in~$G$.  Precise descriptions of uniform and saturable pro-$p$
groups are given in Section~\ref{sec:uniform}.
 

\section{Uniform pro-$p$ groups} \label{sec:uniform} 

\subsection{} Let $G$ be a pro-$p$ group.  The lower central
$p$-series of $G$ is defined as follows: $P_1(G)=G$ and $P_{i+1}(G) =
{P_i(G)}^p [P_i(G), G]$ for $i \in \N$.  We recall the definition of a
uniformly powerful pro-$p$ group.

\begin{definition} \label{def:uniform} A pro-$p$ group $G$ is
  \emph{uniformly powerful}, or \emph{uniform} for short, if
  \begin{itemize}
  \item[(i)] $G$ is finitely generated;
  \item[(ii)] $G$ is powerful, i.e., $[G,G] \subseteq G^p$ if $p \geq
    3$, and $[G,G] \subseteq G^4$ if $p=2$;
  \item[(iii)] $\lvert P_i(G) : P_{i+1}(G) \rvert = \lvert G : P_2(G)
    \rvert$ for all $i \in \N$.
 \end{itemize}
\end{definition}

A useful characterisation of uniform pro-$p$ groups is the following.

\begin{thm}[{\cite[Theorem~4.5]{DidSMaSe99}}] \label{thm:uniform-iff-pow-tf}
  A pro-$p$ group is uniform if and only if it is finitely generated,
  torsion-free and powerful.
\end{thm}

If a finitely generated pro-$p$ group $G$ is powerful, then it has
rank $\rank(G) = d(G)$, but the converse does not generally hold; see
\cite[Theorem~3.8]{DidSMaSe99}.  In view of \cite[Theorem~1.3]{Kl11},
we can rephrase Conjecture~\ref{con:infinite} as follows.

\begin{conjecture-prime}
  Suppose that $p \geq 3$ and let $G$ be a pro-$p$ group.  Then $G$ is
  uniform if and only if it is finitely generated, torsion-free and
  $\rank(G) = d(G)$.
\end{conjecture-prime}

\subsection{} We now prove the results about soluble pro-$p$ groups
stated in Section~\ref{sec:main-results}.  Furthermore, we provide
examples to explain the restrictions that we impose on~$p$.

\begin{proof}[Proof of Theorem~\ref{thm:soluble}]
  Suppose that $d(G) = \dim(G)$.  We need to prove that $G$ is
  powerful and torsion-free.  If $G$ is the trivial group there is
  nothing further to do.  Hence suppose that $G \neq 1$.

  First we show that $G$ is powerful.  Choose a normal subgroup $N$ of
  $G$ such that $H := G/N$ is just-infinite.  Note that both $\pi(N)$
  and $\pi(H)$ are trivial.  By \cite[Theorem~1.3]{Kl11} we have
  $d(N)\leq \dim(N)$ and $d(H)\leq \dim(H)$.  Thus we deduce from
  \cite[Theorem~4.8]{DidSMaSe99} that
  \begin{displaymath}
    \dim(G) = d(G) \leq d(H)+d(N) \leq \dim(H) + \dim(N) = \dim(G),
  \end{displaymath}
  which implies $d(H)=\dim(H)$ and $d(N)=\dim(N)$.  Since $\dim(N) <
  \dim(G)$, it follows by induction that $N$ is powerful.  We observe
  that in order to show that $G$ is powerful it suffices to show that
  $H$ is powerful: if $H$ is powerful then
  \begin{multline*}
    \lvert G : G^p \rvert = \lvert G : G^p N \rvert \lvert N : N
    \cap{G}^p \rvert \leq \lvert H : H^p\rvert \lvert N : N^p \rvert =
    p^{d(H) + d(N)} \\ = p^{\dim(H)+ \dim(N)} = p^{\dim(G)} = p^{d(G)}
    = \lvert G : G^p[G,G] \vert \leq \lvert G : G^p \rvert,
  \end{multline*}
  and we obtain $[G,G] \leq G^p$.

  Therefore we may assume that $G=H$ is just-infinite.  Since $G$ is
  soluble, we deduce that $G$ is virtually abelian;
  see~\cite[Ch.~12]{LeMc2002}.  Put $d = d(G) = \dim(G)$ and
  choose an open normal subgroup $B \trianglelefteq G$ such that $B
  \cong \mathbb{Z}_p^d$.  Let $A := \Cen_G(B) \trianglelefteq G$, the
  centraliser of $B$ in~$G$, and write $\Zen(A)$ for the centre
  of~$A$.  Then $\lvert A : \Zen(A) \rvert \leq \lvert A : B \rvert <
  \infty$, and hence $[A,A]$ is finite by Schur's theorem.  Since $G$
  is just-infinite we must have $[A,A]=1$.  Hence $A$ is abelian and
  self-centralising in~$G$.  Since $\pi(G)=1$, we conclude that $A$ is
  torsion-free.  The group $\bar{G} := G/A$ is finite and acts
  faithfully on $A \cong \Z_p^d$.  In this way we obtain an embedding
  $\bar{G} \hookrightarrow \GL(A) \cong \GL_d(\Z_p)$.  If $\bar{G}$ is
  trivial then $G = A$ is abelian, hence powerful.

  For a contradiction, we now assume that $\bar{G} \neq 1$.  Let $C =
  \langle x \rangle A$ be a subgroup of $G$ such that $\bar{C} := C/A
  = \langle \bar{x} \rangle$ is cyclic of order $p$ and contained in
  the centre $\Zen(\bar{G})$ of $\bar{G}$.  According
  to~\cite[Theorem~2.6]{HeRe62}, there are three indecomposable types
  of $\Z_p \bar{C}$-modules which are free and of finite rank as
  $\Z_p$-modules:
  \begin{enumerate}
  \item[(i)] the trivial module $I = \Z_p$ of $\Z_p$-dimension $1$,
  \item[(ii)] the module $J = \Z_p \bar{C} / (\Phi(\bar{x}))$ of
    $\Z_p$-dimension $p-1$, where $\Phi(X) = 1 + X + \ldots + X^{p-1}$
    denotes the $p$th cyclotomic polynomial,
  \item[(iii)] the free module $K = \Z_p \bar{C}$ of $\Z_p$-dimension
    $p$.
  \end{enumerate}
  Hence the $\Z_p \bar{C}$-module $A$, which is free and of finite
  rank as a $\Z_p$-module, decomposes as a direct sum of
  indecomposable submodules
  \begin{displaymath}
    A = (\oplus_{i=1}^{m_1}I_i) \oplus (\oplus_{j=1}^{m_2}J_j) \oplus
    (\oplus_{k=1}^{m_3}K_k), 
  \end{displaymath}
  where $m_1, m_2, m_3 \in \{0 \} \cup \N$ and $I_i\cong I$, $J_j\cong
  J$, $K_k\cong K$ for all indices $i,j,k$.

  Put $A_1 = \oplus_{i=1}^{m_1}I_i$, $A_2 = \oplus_{j=1}^{m_2}J_j$ and
  $A_3 = \oplus_{k=1}^{m_3}K_k$.  Since $\bar{C}$ is central in
  $\bar{G}$, the decomposition $A = A_1\oplus A_2 \oplus A_3$ is
  $\bar{G}$-invariant. Since $G$ is just-infinite we conclude that $A
  = A_i$ for precisely one $i \in \{1,2,3\}$.  We cannot have $A=A_1$,
  because $\bar{C}$ acts faithfully on $A$.  Thus either $A = A_2$ or
  $A = A_3$, and consequently $d_{\Z_p \bar{C}}(A) \leq \max \{ m_2,
  m_3 \} \leq \lfloor \frac{d}{p-1} \rfloor$.  Moreover,
  \cite[Proposition~3.5]{Kl11} shows that $d(\bar{G})\leq \lfloor
  \frac{d}{p-1} \rfloor$.  Hence from $p \geq 5$ we obtain
  \begin{equation} \label{equ:p-geq-5}
    d(G)\leq d(\bar{G}) + d_{\Z_p \bar{C}}(A) \leq 2 \lfloor
    d/(p-1) \rfloor < d = \dim(G)
  \end{equation}
  in contradiction to $d(G) = \dim(G)$.  This concludes the proof that
  $G$ is powerful.
  
  It remains to show that $G$ is torsion-free.  By
  \cite[Theorem~4.20]{DidSMaSe99}, the collection of all elements of
  finite order in $G$ forms a characteristic subgroup $T$ of~$G$.
  Clearly, $T \subseteq \pi(G)=1$ implies that $G$ is torsion-free.
\end{proof}

The following example illustrates that the assertion of
Theorem~\ref{thm:soluble} does not extend without modifications to
$p=3$.

\begin{example}\label{one}
  Consider the pro-$3$ group $G = \langle z \rangle \ltimes
  \Z_3[\xi]$, where $\langle z \rangle \cong C_3$, $\Z_3[\xi] = \Z_3 +
  \Z_3 \xi \cong \Z_3^2$ for a primitive $3$rd root of unity $\xi$ and
  where $z$ acts on $\Z_3[\xi]$ as multiplication by~$\xi$.  One
  easily verifies that $G$ is not powerful, even though $G$ is
  soluble, $\pi(G)=1$ and $d(G) = 2 = \dim(G)$.
\end{example}

In order to prove Theorem~\ref{thm:soluble-p-equals-3} we need to
analyse more carefully the case $p=3$.  

\begin{lem} \label{three}
  Suppose that $p \geq 3$ and let $G$ be a just-infinite soluble
  pro-$p$ group of finite rank such that $\pi(G)=1$. If $d(G) =
  \dim(G)$, then $\dim(G) \leq 2$.
\end{lem}

\begin{proof}
  Suppose that $d := d(G) = \dim(G)$.  If $G$ is abelian, then $G
  \cong \Z_p$ and $d = 1$.  Now suppose that $G$ is not abelian.
  Arguing as in the proof of Theorem~\ref{thm:soluble}, we find an
  open normal subgroup $A \trianglelefteq G$ such that the quotient
  $\bar{G} = G/A \neq 1$ acts faithfully on~$A \cong \Z_p^d$.
  Moreover, there is a central cyclic subgroup $\langle \bar{x}
  \rangle = \bar{C} \leq \bar{G}$ of order~$p$ such that $A$, regarded
  as a $\Z_p \bar{C}$-module, decomposes into a $\bar{G}$-invariant
  homogeneous direct sum of pairwise isomorphic indecomposable
  submodules:
  \begin{displaymath}
    A = \oplus_{j=1}^mJ_j \quad \text{with $d = m (p-1)$} \qquad
    \text{or} \qquad A = \oplus_{k=1}^mK_k \quad \text{with $d = m p$,}  
  \end{displaymath}
  just as in the proof of Theorem~\ref{thm:soluble}.  If $p \geq 5$,
  then \eqref{equ:p-geq-5} yields a contradiction.

  Hence we have $p=3$.  Then $d(\bar{G}) \leq \lfloor d/2 \rfloor$
  by~\cite[Proposition~3.5]{Kl11} and, if $A = \oplus_{k=1}^m K_k$
  with $d = 3m$, then
 \begin{displaymath}
   d(G) \leq d(\bar{G}) + d_{\Z_3 \bar{C}}(A) \leq \lfloor d/2
   \rfloor + d/3 < d=\dim(G), 
  \end{displaymath}
  yielding a contradiction.  Hence we obtain $A = \oplus_{j=1}^mJ_j$
  with $d = 2m$.  We can regard $A$ as a free module of rank $m$ over
  the ring $R = \Z_p \bar{C} / (\bar{x}^2 + \bar{x} + 1)$.  This ring
  is naturally isomorphic to the valuation ring $\Z_3[\xi]$ of the
  totally ramified extension $\Q_3(\xi)$ of $\Q_3$ obtained by
  adjoining a primitive $3$rd root of unity~$\xi$.  The element $\pi$
  represented by $\bar{x}-1$ is a uniformiser of~$R$.  We obtain an
  embedding $\eta \colon \bar{G} \hookrightarrow \GL(A) \cong
  \GL_m(R)$.  The Sylow pro-$3$ subgroups of $\GL_m(R)$ are all
  conjugate to one another and we may assume that $\bar{G}^\eta$
  consists of matrices which are upper uni-triangular modulo~$\pi$.
  If $\bar{G}^\eta$ is not contained in $\GL_m^1(R) = \ker (\GL_m(R)
  \to \GL_m(R/\pi R))$ then
  \begin{displaymath}
    d(G) \leq d(\bar{G}) + d_{R \bar{G}}(A) \leq d/2 +
    (m - 1) = d-1 < d =\dim(G) 
  \end{displaymath}
  yields a contradiction.

  Hence we obtain $\bar{G}^\eta \subseteq \GL_m^1(R)$.  We observe
  that $\GL_m^2(R) = \ker (\GL_m(R) \to \GL_m(R/\pi^2R))$ is
  torsion-free.  Since $\bar{G}^\eta$ is finite, this implies that
  $\bar{G}$ embeds into $\GL_m^1(R)/\GL_m^2(R)$.  But the latter group
  is elementary abelian, hence $\bar{G}$ is elementary abelian.  Let
  $F \cong \Q_3(\xi)$ denote the field of fractions of $R$.  Since $G$
  is just-infinite, $V := F \otimes_R A$ is an irreducible $F
  \bar{G}$-module.  Since $\bar{G}$ is abelian of exponent~$3$ and $F$
  contains a primitive $3$rd root of unity, we deduce that $m =
  \dim_F(V) = 1$ and $d =2$.
\end{proof}

We remark that Lemma~\ref{three} has no analogue for insoluble groups,
because there are just-infinite uniform pro-$p$ groups of arbitrarily
large dimension.
 
\begin{lem} \label{lem:uniform-by-cyclic}
  Let $G$ be a pro-$p$ group with a powerful open normal subgroup $N
  \trianglelefteq G$, and suppose that $z \in G \setminus N$.  If $G$
  is torsion-free, then $\langle z \rangle \cap (N \setminus N^p) \neq
  \varnothing$.
\end{lem}

\begin{proof}
  Suppose that $G$ is torsion-free.  Without loss of generality we may
  assume that $z^p \in N$ and, for a contradiction, we assume that
  $z^p \in N^p$.  Note that $N$ is uniform so that $z^p = a^p$ for $a
  \in N$.  Furthermore, $N$ admits the structure of a $\Z_p \langle z
  \rangle$-module which is free and of finite rank over $\Z_p$; see
  \cite[Theorem~4.17]{DidSMaSe99}.  Thus $a^p \in \Cen_N(z)$ implies
  $a \in \Cen_N(z)$ so that $(z a^{-1})^p = 1$.  Since $z a^{-1} \neq
  1$, this yields a contradiction.
\end{proof}

\begin{proof}[Proof of Theorem~\ref{thm:soluble-p-equals-3}] Suppose
  that $d(G) = \dim(G)$.  If $G$ is trivial, there is nothing to
  prove.  Now suppose that $G \neq 1$ and choose a normal subgroup $N
  \trianglelefteq G$ such that $H := G/N$ is just-infinite. Proceeding
  as in the proof of Theorem~\ref{thm:soluble}, we conclude that $d(N)
  = \dim(N)$ and $d(H) = \dim(H)$.  By induction, $N$ is uniform.
  Moreover, since $\pi(H)=1$, Lemma~\ref{three} implies that $\dim(H)
  \leq 2$.

  For a contradiction, assume that $\dim(H)=2$.

  \smallskip

  \noindent \emph{Claim.} We have $H \cong \langle \bar{z} \rangle\
  \ltimes \Z_3[\xi]$, where $\langle \bar{z} \rangle \cong C_3$,
  $\Z_3[\xi] = \Z_3 + \Z_3 \xi \cong \Z_3^2$ for a primitive $3$rd
  root of unity $\xi$ and where $\bar{z}$ acts on $\Z_3[\xi]$ as
  multiplication by $\xi$.

  \smallskip

  \noindent \emph{Proof of the claim.} We analyse $H$ as before. Let
  $A\cong \Z_3^2$ be a self-centralising normal subgroup of $H$. Then
  $H/A$ acts faithfully on $A$ so that we obtain an embedding $\bar{H}
  = H/A \hookrightarrow \GL_2(\Z_3)$.  Inspection of the Sylow pro-$3$
  subgroup of $\GL_2(\Z_3)$ shows that $\bar{H} \cong C_3$ is acting
  fixed-point-freely on~$A$.  Thus $H = \langle \bar{z} \rangle
  \ltimes A$, where ${\bar{z}}$ has order $3$ and acts on $A\cong
  \Z_3[\xi]$ as multiplication by a $3$rd root of~$1$. This proves the
  claim.

  \smallskip

  In particular, $d(H) = 2$ and we can complement $\bar{z}$ to a
  generating pair for~$H$.  Choose a pre-image $z \in G$ of $\bar{z}$
  with respect to the quotient map $G \to G/N$.  Since $G$ is
  torsion-free, the powerful group $N$ is uniform and
  Lemma~\ref{lem:uniform-by-cyclic} implies that $z^3 \in N \setminus
  N^3 = N \setminus N^3[N,N]$.  This implies that
  \begin{displaymath}
    d(G) \leq d(H) + (d(N) - 1) = \dim(H) + \dim(N) - 1 < \dim(G),
  \end{displaymath}
  which is a contradiction.

  This forces $\dim(H)=1$ so that $H \cong \Z_3$ is powerful.  From
  \begin{displaymath}
    3^{\dim(G)} = 3^{d(G)} \leq \lvert G:G^3 \rvert \leq 3 \lvert N:N^3 \rvert =
    3^{1+\dim(N)} = 3^{\dim(G)} 
  \end{displaymath}
  it follows that $\lvert G:G^3[G,G] \rvert = 3^{d(G)}= \lvert G:G^3
  \rvert$. We deduce that $[G,G] \subseteq G^3$.  Hence $G$ is
  powerful and, because $G$ is torsion-free, $G$ is uniform.
\end{proof}

Implicitly the proofs of Theorems~\ref{thm:soluble} and
\ref{thm:soluble-p-equals-3} also yield the following fact.

\begin{pro}
  Suppose that $p \geq 3$ and let $G$ be a soluble uniform pro-$p$
  group.  Then $G$ is poly-$\Z_p$, i.e., there exists a finite chain
  of subgroups $G = G_1 \supseteq \ldots \supseteq G_{r+1} = 1$ such that
  $G_{i+1} \trianglelefteq G_i$ and $G_i/G_{i+1} \cong \Z_p$ for $1
  \leq i \leq r$.
\end{pro}

\begin{example}
  Consider the metabelian group $G := \langle y \rangle \ltimes A$,
  where $\langle y \rangle \cong \Z_2$, $A \cong \Z_2^{d-1}$ with $d
  \geq 2$ and $y$ acts on $A$ as scalar multiplication by $-1$.  Note
  that $G$ is not powerful, even though it is torsion-free and $d(G) =
  \dim(G) = \rank(G)$.
\end{example}

\begin{proof}[Proof of Corollary~\ref{cor:nilpotent}]
  By \cite[Proposition~16.4.2(v)]{LuSe03} the pro-$p$ completion
  $\widehat{\Gamma}_p$ of $\Gamma$ is a torsion free pro-$p$ group of
  finite rank and $\dim(\widehat{\Gamma}_p) = h(\Gamma)$.  Clearly,
  $\widehat{\Gamma}_p$ is soluble.  Now, using
  Theorems~\ref{thm:soluble} and~\ref{thm:soluble-p-equals-3}, we
  deduce that $\widehat{\Gamma}_p$ is a uniform pro-$p$ group if and
  only if $d(\widehat{\Gamma}_p) = \dim(\widehat{\Gamma}_p)$. Since
  \begin{displaymath}
    d(\widehat{\Gamma}_p) = \dim_{\F_p}(\widehat{\Gamma}_p /
     (\widehat{\Gamma}_p)^p [\widehat{\Gamma}_p,\widehat{\Gamma}_p] ) =
    \dim_{\F_p}(\Gamma / \Gamma^p [\Gamma,\Gamma] ),
  \end{displaymath}
  the assertion follows.
\end{proof}

Let $\Gamma$ be a finitely generated nilpotent group.  If
$\Gamma/[\Gamma,\Gamma]$ and $\Gamma$ have the same Hirsch length,
i.e., $h(\Gamma/[\Gamma,\Gamma]) = h(\Gamma)$, then $[\Gamma, \Gamma]$
is finite and $\Gamma$ is an FC-group, i.e., a group with finite
conjugacy classes, so that $\lvert \Gamma : \Zen(\Gamma) \rvert <
\infty$.  In this case $\widehat{\Gamma}_p$ is uniform for almost all
primes~$p$.  On the other hand, if $h(\Gamma/[\Gamma,\Gamma]) <
h(\Gamma)$ then $\widehat{\Gamma}_p$ is uniform for at most finitely
many primes $p$: the group fails to be uniform for all $p$ with $p
\notin \Pi(\Gamma/[\Gamma,\Gamma])$.  This should be contrasted with
the fact that in any case $\widehat{\Gamma}_p$ is saturable for almost
all primes~$p$, because it is so for all $p \notin \Pi(\Gamma)$ with
$p > h(\Gamma)$, by~\cite[Theorem~A]{GoKl09}.

\subsection{} Next we prove the assertions in
Section~\ref{sec:main-results} about pro-$p$ groups $G$ of finite rank
which are not necessarily soluble. We recall the following concepts
from~\cite{FeGoJa08}.

\begin{definition} \label{def:potent} Let $G$ be a pro-$p$ group. A
  \emph{potent filtration} in $G$ is a descending series $N_i$, $i \in
  \N$, of subgroups of $G$ satisfying the following conditions:
  \begin{itemize}
  \item[(i)] $\bigcap_{i\in \mathbb{N}} N_i = 1$,
  \item[(ii)] $[N_i, G] \leq N_{i+1}$ for all $i \in \N$,
  \item[(iii)] $[N_i,\underbrace{G,\ldots,G}_{p-1}] \leq N_{i+1}^p$ for
    all $i \in \N$.
  \end{itemize}

  We say that a subgroup $N$ of $G$ is \emph{PF-embedded} in $G$ if
  there is a potent filtration in $G$ starting at $N$.  The pro-$p$
  group $G$ is called a \emph{PF-group} if $G$ is PF-embedded in
  itself.
\end{definition}

Note that if $N_i$, $i\in \N$, is a potent filtration in a pro-$p$
group $G$ then for each $k\in \N$ the series $N_i$, $i \geq k$, is a
potent filtration starting at~$N_k$.  In particular, each $N_k$ is a
PF-group.

A finitely generated pro-$p$ group is saturable if it admits a certain
type of `valuation map'; for precise details we refer to \cite{Kl05}.
For instance, if $G$ is a uniform pro-$p$ group then one can show that
$G$ is saturable by considering the valuation map
\begin{displaymath}
  \omega \colon  G \rightarrow \mathbb{R}_{>0} \cup
  \{ \infty \}, \quad x \mapsto \sup \{ k \mid k \geq 1 \text{ and } x
  \in G^{p^{k-1}} \}. 
\end{displaymath}
In \cite{Go07}, Gonz\'alez-S\'anchez proved that a finitely generated
pro-$p$ group $G$ is saturable if and only if it is a torsion-free
PF-group.



\begin{proof}[Proof of Theorem~\ref{thm:injective}]
  Suppose that $d(G) = \dim(G)$.  Clearly, $G$ is torsion-free and it
  suffices to prove that $G$ is powerful.  Let $N$ be an open
  PF-embedded subgroup of $G$ and let $N_i$, $i \in \N$, be a potent
  filtration in $G$ starting at $N_1 = N$.
  
  For any subset $X \subseteq G$ we denote by $X^{\{p\}}$ the
  collection of all $p$th powers $x^p$ of elements $x \in X$.  Since
  $N$ is a saturable pro-$p$ group, we have $N^p=N^{\{ p \}}$.
       
  \smallskip
 
  \noindent \emph{Claim.} For every $x\in G$ we have $(xN)^{\{p\}} =
  x^p N^p$.

  \smallskip

  \noindent \emph{Proof of the claim.}  P.\ Hall's collection formula
  shows that for all $x\in G$ and $u\in N$,
  \begin{displaymath} 
    (xu)^p \equiv x^p u^p \mod 
    \gamma_2(\langle x, u \rangle)^p \gamma_p(\langle x, u \rangle).
  \end{displaymath}
  
  Since $N_i$, $i \in \N$, is a potent filtration in $G$ it follows
  that for all $x\in G$ and $w\in N_i$ we have $\gamma_2(\langle x, w
  \rangle)^p \subseteq [N_i,G]^p \subseteq N_{i+1}^{p}$ and
  $\gamma_p(\langle x, w \rangle) \subseteq
  [N_i,\underbrace{G,\ldots,G}_{p-1}] \subseteq N_{i+1}^{p}$. This
  implies that
  \begin{equation}\label{eq:congruence}
    (xw)^p \equiv x^p w^p \mod N_{i+1}^p
  \end{equation} 
  for all $x\in G$ and $w\in N_i$. Now since $N=N_1$, we deduce that
  $(xN)^{\{p\}} \subseteq x^pN^p$.
  
  For the reverse inclusion, we need to show that, given $x\in G$ and
  $u\in N$, there exists $v \in N$ such that $(xv)^p = x^pu^p$.  Note
  that by \eqref{eq:congruence} we have $(xu)^p u_2^p = x^pu^p$ for
  some $u_2\in N_2$.  Applying once more \eqref{eq:congruence}, we
  obtain $ (x u u_2)^p u_3^p = x^p u^p$ for some $u_3 \in
  N_3$. Proceeding in this way, we construct inductively a sequence
  $(u_i)_{i = 1}^\infty$, with $u_1 = u$, such that $u_i \in N_i$ and
  $(x u_1 u_2 \cdots u_i)^p u_{i+1}^p = x^pu^p$ for every $i \in
  \N$.  Writing $v_i = u_1 u_2 \cdots u_i$, we see that
  $(v_i)_{i= 1}^\infty$ is a Cauchy sequence and, taking limits, we
  obtain $(xv)^p = x^p u^p$, where $v = \displaystyle\lim_{i \to
    \infty} v_i \in N$.  This proves the claim.

  \smallskip

  Now suppose that $x,y \in G$ with $x^p N^p = y^p N^p$.  Then there
  exists $u \in N$ such that $(xu)^p = y^p$.  Since the $p$-power map
  is injective on $G$, we conclude that $xu = y$ so that $xN = yN$.
  Hence $G \rightarrow G$, $x \mapsto x^p$ induces a bijective
  correspondence between the cosets of $N$ in $G$ and their $p$th
  powers, i.e., a bijection $\{ xN \mid x \in G \} \rightarrow \{ x^p
  N^p \mid x \in G \}$, $x N \mapsto (xN)^{\{p\}} = x^p N^p$.

  Let $\mu$ be the normalised Haar measure on the compact group~$G$ so
  that $\mu(G) = 1$. Since $N$ is a saturable open subgroup of $G$,
  note that $\lvert N : N^p \rvert = p^{\dim(N)} = p^{\dim(G)}$. For
  every $x \in G$ we see that
  \begin{displaymath}
    \mu(x^p N^p) = \mu(N^p) = \lvert N : N^p \rvert^{-1} \mu(N) =
    p^{-\dim(G)} \lvert G : N \rvert^{-1},
  \end{displaymath}
  and thus we conclude that
  \begin{displaymath}
    \mu(G^{\{ p \}}) = \lvert G:N \rvert \left( p^{-\dim(G)} \lvert
      G:N \rvert^{-1} \right ) = p^{-\dim(G)}.  
  \end{displaymath}
  Using $d(G)=\dim(G)$, this yields
  \begin{displaymath}
    p^{d(G)} = \lvert G: G^p [G,G] \rvert \leq \lvert G:G^p
    \rvert = (\mu(G^p))^{-1} \leq \left( \mu(G^{\{p\}} \right)^{-1} =
    p^{\dim(G)} = p^{d(G)}.
  \end{displaymath}
  Hence $[G,G] \subseteq G^p$, and $G$ is powerful.
\end{proof}

\begin{proof}[Proof of Corollary~\ref{cor:sat-un}] 
  Suppose that $G$ is a subgroup of a finitely generated saturable
  pro-$p$ group~$S$.  Since $S$ is saturable, the map $S \rightarrow
  S$, $x \mapsto x^p$ is injective and therefore its restriction to
  $G$ is also injective.  By Theorem~\ref{thm:injective} it suffices
  to show that $G$ has an open PF-embedded subgroup.
  
  The saturable closure of $G$ in $S$, denoted by $\sat_S(G)$, is a
  saturable subgroup of $S$ which contains $G$ as an open subgroup and
  is the smallest saturable subgroup with this property;
  see~\cite{La65} or~\cite{Kl05}.  Replacing $S$ by $\sat_S(G)$, we
  may assume that $\lvert S :G \rvert$ is finite. Since $S$ is
  saturable, it is a PF-group. Let $S_i$, $i \in \N$, be a potent
  filtration of $S$ starting at $S_1 = S$.  Since $\lvert S : G \rvert
  < \infty$, there exists a positive integer $k$ such that
  $S^{p^k}\subseteq G$. Then $S^{p^k}$ is an open subgroup of $G$ and
  \cite[Proposition~3.2(iii)]{FeGoJa08} implies that $S_i^{p^k}$, $i
  \in \N$, is a potent filtration of~$S$. In particular, $S^{p^k}$ is
  an open PF-embedded subgroup of~$G$.
\end{proof}

As mentioned in the introduction, saturable pro-$p$ groups need not be
uniform.  Corollary~\ref{cor:char-sat-uniform}, which provides a
practical criterion for deciding whether a saturable group is uniform,
is a direct consequence of Corollary~\ref{cor:sat-un}.

\begin{proof}[Proof of Corollary~\ref{cor:large-p}]
  In~\cite{GoKl09} Gonz\'alez-S\'anchez and Klopsch proved that every
  torsion-free $p$-adic analytic pro-$p$ group of dimension less than
  $p$ is saturable.  Thus the assertion follows from
  Corollary~\ref{cor:char-sat-uniform}.
\end{proof}

Suppose that $p\geq3$. We recall that a subgroup $N$ of $G$ is
powerfully embedded in $G$ if $[N,G] \subseteq N^p$.  If $N$ is
powerfully embedded in $G$ then $N^{p^{i-1}}$, $i \in \N$, is a potent
filtration in~$G$; in particular, $N$ is PF-embedded in~$G$.  Thus we
obtain from Theorem~\ref{thm:injective} also the following corollary.

\begin{cor} \label{cor:powerfully embedded} Suppose that $p\geq
  3$. Let $G$ be a pro-$p$ group of finite rank with an open
  powerfully embedded subgroup and such that the map $G \to G$, $x
  \mapsto x^p$ is injective.  If $d(G) = \dim(G)$ then $G$ is uniform.
\end{cor}

\subsection{} Finally, we explore a possible analogue of
Conjecture~\ref{con:infinite} for finite $p$-groups, as suggested
toward the end of Section~\ref{sec:main-results}.

\begin{proof}[Proof of Proposition~\ref{pro:qu-implies-conj}]
  Let $p\geq 3$ and suppose that the answer to
  Question~\ref{con:finite} is `yes'.  Let $G$ be a torsion-free
  pro-$p$ group of finite rank such that $d(G) = \dim (G)$.  We need
  to prove that $G$ is powerful.

  Let $U$ be a uniform open normal subgroup of $G$.  The open
  subgroups $U^{p^n}$, $n \in \N$ form a base for the neighbourhoods
  of $1$ in~$G$.  By Lemma~\ref{lem:uniform-by-cyclic} we have that
  $\{ x\in G \mid x^p\in U^p \} \subseteq U$. Since $U$ is uniform,
  this implies that for all $n\in \mathbb{N}$,
  \begin{equation} \label{equ:uniform} \Omega_1(G/U^{p^n})=\langle
    xU^{p^n} \in G/U^{p^n} \mid x^p \in U^{p^n} \rangle =
    U^{p^{n-1}}/U^{p^n}
  \end{equation}
  has size $p^{\dim(U)}$.  Since $p \geq 3$, we conclude from a result
  of Laffey~\cite[Corollary~2]{Laf73} that
  \begin{displaymath}
    \dim(G) = d(G)= \limsup_{n\in \mathbb{N}}d(G/U^{p^n})\leq
    \limsup_{n\in \mathbb{N}} \log_p \lvert \Omega_1(G/U^{p^n}) \rvert = \dim(G). 
  \end{displaymath}
  This implies that $d(G/U^{p^n}) = \log_p \lvert \Omega_1(G/U^{p^n})
  \rvert = d(G)$ for infinitely many $n \in \N$.  Let $k$ be the
  smallest positive integer such that $d(G/U^{p^k})=d(G)$. Then for
  each $n\geq k$ we have $d(G)\geq d(G/U^{p^n})\geq
  d(G/U^{p^k})=d(G)$, and consequently, $d(G/U^{p^n})=d(G)=\log_p
  \lvert \Omega_1(G/U^{p^n}) \rvert$.  The positive answer to
  Question~\ref{con:finite} implies that $G/U^{p^n}$ is a powerful
  finite $p$-group for each $n \geq k$.  Since $G$ is the inverse
  limit of the groups $G/U^{p^n}$, $n\geq k$, we deduce from
  \cite[Corollary~3.3]{DidSMaSe99} that $G$ is powerful.
\end{proof}

We remark that the group $\Omega_1(G/U^{p^n})=U^{p^{n-1}}/U^{p^n}$ in
\eqref{equ:uniform} is an elementary abelian $p$-group.  The proof of
Proposition~\ref{pro:qu-implies-conj} shows that
Conjecture~\ref{con:infinite} is true whenever the following `weaker'
version of Question~\ref{con:finite} has a positive answer.

\begin{question} \label{con:finite1} Suppose that $p \geq 3$ and let
  $G$ be a finite $p$-group such that $\Omega_1(G)$ is an elementary
  abelian $p$-group.  Is it true that $G$ is powerful if and only if
  $d(G) = \log_p \lvert \Omega_1(G) \rvert$?
\end{question}

Using \cite[Corollary~2]{Laf73}, it is easy to see that
Question~\ref{con:finite} has positive answer for all finite
$p$-groups $G$ with the property $\lvert G:G^p \rvert = \lvert
\Omega_1(G) \rvert$.  In the 1930s P.~Hall showed that this condition
is satisfied for regular $p$-groups; see~\cite{Ha33}.  It was proved
more recently in \cite{GoJa04} that for $p \geq 3$ the same property
is shared by every finite $p$-group $G$ which is potent, i.e., which
satisfies $\gamma_{p-1}(G) \leq G^p$; indeed, in this situation it is
true that $\lvert N:N^p \rvert = \lvert \Omega_1(N) \rvert$ for every
normal subgroup $N$ of $G$.
  
Suppose that $p\geq 3$ and consider a finite $p$-group $G$ with
$\Omega_1(G) \leq Z(G)$; such a group is called a $p$-central group.
Suppose that $d(G) = \log_p \lvert \Omega_1(G) \rvert$. Since $G$ is
$p$-central, by~\cite[Proposition~4]{Ma07}, we have $\lvert G:G^p
\rvert \leq \lvert \Omega_1(G) \rvert$.  Now from
\begin{displaymath}
  \lvert \Omega_1(G) \rvert = p^{d(G)} = \lvert G : G^p[G,G] \rvert \leq
  \lvert G:G^p \rvert \leq \lvert \Omega_1(G) \rvert
\end{displaymath}
we deduce that $[G,G] \subseteq G^p$, which means that $G$ is
powerful. Hence for $p$-central groups the answer to
Question~\ref{con:finite} is positive.

\subsection{}
We conclude the section with two observations.  In~\cite{Kl11} Klopsch
proved for $p \geq 3$ that every pro-$p$ group $G$ of finite rank with
$\pi(G)=1$ satisfies $\rank(G)=\dim(G)$.  This allows us to replace
$\dim(G)$ by $\rank(G)$ in all the results of this section.  For
instance we obtain the following consequence.

\begin{cor}
  Suppose that $p \geq 3$. Let $G$ be a finitely generated pro-$p$
  group with an open PF-embedded subgroup and such that the map $G \to
  G$, $x \mapsto x^p$ is injective.  If $d(G) = \rank(G)$ then $G$ is
  uniform.
\end{cor}

Secondly, we record a straightforward, but useful result, which can be
regarded as a modification of~\cite[Proposition~4.4]{DidSMaSe99}.

\begin{pro}
  Let $G$ be a finitely generated powerful pro-$p$ group such that
  $d(G) = \dim(G)$.  Then $G$ is uniform.
\end{pro}

\begin{proof}
  Let $G_i=P_i(G)$, $i \in \N$, denote the terms of the lower central
  $p$-series of~$G$, and put $d_i := \log_p \lvert G_i:G_{i+1}
  \rvert$. By \cite[Theorem~3.6(iv)]{DidSMaSe99}, we have $d_1 \geq
  d_2 \geq \ldots$, hence there exists $k \in \N$ such that $d_i =
  d_k$ for all $ i \geq k$.  Now, by
  \cite[Theorem~3.6(ii)]{DidSMaSe99}, we obtain $P_j(G_k) = G_{k+j-1}
  = G_k^{p^j} = P_{j+k-1}(G)$ for all $j \in \N$, and
  \cite[Theorem~3.6(i)]{DidSMaSe99} shows that $G_k$ is powerful.
  Moreover, we have $\lvert P_j(G_k) : P_{j+1}(G_k) \rvert = \lvert
  G_k:P_2(G_k) \rvert$ for each $j \in \N$, which means that $G_k$ is
  uniform. Now we have
  \begin{displaymath}
    d(G)= d_1 \geq d_2 \geq \ldots \geq d_k = d(G_k) = \dim(G) = d(G),
  \end{displaymath}
  and consequently $d_i = d_1$ for all $i \in \N$. Hence $G$ is
  uniform.
\end{proof}


\section{Hereditarily powerful pro-$p$
  groups} \label{sec:hereditarily}

In the present section we prove the assertions in
Section~\ref{sec:applications}.

\begin{proof}[Proof of Theorem~\ref{thm:her-powerful}]
  Suppose that $G$ is hereditarily powerful.  Then $G$ is the inverse
  limit of hereditarily powerful finite $p$-groups $G_i$, $i \in I$,
  with respect to connecting homomorphisms $\phi_{ij} \colon G_i \to
  G_j$ for $i \succeq j$, where $I$ is a suitable directed set.  By
  \cite[Theorems~3.1 and~4.3.1]{LuMa87I}, a finite $p$-group is
  hereditarily powerful if and only if it is modular and, if $p=2$,
  not Hamiltonian.  (Thus excluded are all direct products of the
  quaternion group $Q_8$ with elementary abelian $2$-groups.)

  Finite modular groups were classified by Iwasawa; see \cite{Iw41} or
  \cite[Theorem~2.3.1]{Sc94}.  According to this classification, every
  finite $p$-group $H$ which is hereditarily powerful is of the
  following form: $H$ contains an abelian normal subgroup $K$ such
  that $H/K$ is cyclic and there exist an element $h \in H$ with $H =
  \langle h \rangle K$ and a positive integer $s$ such that $h^{-1}kh
  = k^{1+p^s}$ for all $k \in K$, with $s\geq 2$ in case $p=2$.  Hence
  each $G_i$, $i \in I$, is of this form and we denote by $X_i$ the
  non-empty, finite set consisting of all triples $(A_{i \lambda},b_{i
    \lambda},s_{i \lambda})$ such that $G_i = \langle b_{i \lambda}
  \rangle A_{i \lambda}$ with $A_{i \lambda} \trianglelefteq G_i$
  abelian, $s_{i \lambda} \in \{1, 2, \ldots, \log_p \lvert G_i \rvert
  \}$ and $b_{i \lambda} \in G_i$ acting on $A_{i \lambda}$ as
  multiplication by $1+p^{s_{i \lambda}}$ (if $p=2$ we also require
  $s_{i \lambda} \geq 2$).

  We consider the inverse system of the $X_i$, $i\in I$, with respect
  to the connecting maps $X_i \to X_j$ for $i \succeq j$ induced by
  the homomorphisms $\phi_{ij} \colon G_i \to G_j$.  By compactness,
  the inverse limit $X = \varprojlim_{i \in I} X_i$ is non-empty and
  any $x = (A,b,s) \in X$ yields $A \trianglelefteq G$ abelian, $b \in
  G$ and $s \in \N \cup \{\infty\}$ such that $G =
  \langle b \rangle A$ with $b$ acting on $A$ as multiplication by
  $1+p^s$.  Moreover, we obtain $s\geq 2$ for $p=2$.
 
  The proof that groups of the described shape are hereditarily
  powerful is straightforward.
\end{proof}

\begin{proof}[Proof of Corollary~\ref{cor:hered-uniform}]
  By Theorem~\ref{thm:uniform-iff-pow-tf}, a finitely generated
  pro-$p$ group is hereditarily uniform if and only if it is
  hereditarily powerful and torsion-free.

  Suppose that $G = \langle b \rangle A$ is as described in
  Theorem~\ref{thm:her-powerful} and, in addition, torsion-free.  Put
  $d = \dim(G)$.  If $G$ is abelian then $G \cong \Z_p^d$.  If $G$ is
  non-abelian then, because $A$ has elements of infinite order,
  $\langle b \rangle \cap A = 1$ so that $G$ is a semidirect product
  of $\langle b \rangle \cong \Z_p$ and $A \cong \Z_p^{d-1}$.
\end{proof}

In \cite[Theorem~1.1]{KlSn11} we classified finitely generated pro-$p$
groups with constant generating number on open subgroups, that is,
pro-$p$ groups $G$ with the property $d(H) = d(G)$ for every open
subgroup $H \leq G$; see also \cite{Sn09}.

\begin{thm}[Klopsch and Snopce] \label{thm:Kl-Sn}
  Let $G$ be a finitely generated pro-$p$ group and let $d :=
  d(G)$. Then $G$ has constant generating number on open subgroups if
  and only if it is isomorphic to one of the groups in the following
  list:
  \begin{enumerate}
  \item the abelian group $\Z_p^d$, for $d \geq 0$;
  \item the metabelian group $\langle y \rangle \ltimes A$, for $d
    \geq 2$, where $\langle y \rangle \cong \Z_p$, $A \cong
    \Z_p^{d-1}$ and $y$ acts on $A$ as scalar multiplication by
    $\lambda$, with $\lambda = 1+p^s$ for some $s \geq 1$, if $p>2$,
    and $\lambda = \pm (1 + 2^s)$ for some $s \geq 2$, if $p=2$;
  \item the group $\langle w \rangle \ltimes B$ of maximal class, for
    $p=3$ and $d=2$, where $\langle w \rangle \cong C_3$, $B = \Z_3 +
    \Z_3 \omega \cong \Z_3^2$ for a primitive $3$rd root of unity
    $\omega$ and where $w$ acts on $B$ as multiplication by $\omega$;
  \item the metabelian group $\langle y \rangle \ltimes A$, for $p=2$
    and $d \geq 2$, where $\langle y \rangle \cong \Z_2$, $A \cong
    \Z_2^{d-1}$ and $y$ acts on $A$ as scalar multiplication by $-1$.
   \end{enumerate}
\end{thm}

Note that, if $G$ is a hereditarily uniform pro-$p$ group, then $d(H)
= d(G)$ for all open subgroups $H \leq G$.  Hence,
Corollary~\ref{cor:hered-uniform} can also be regarded as a consequence
of Theorem~\ref{thm:Kl-Sn}.

Conversely, Corollary~\ref{cor:hered-uniform} can be used to give a new
proof of Theorem~\ref{thm:Kl-Sn}, at least in the case $p \geq 3$.
Indeed, the argument given in~\cite{KlSn11} proceeds by induction on
the index of a saturable open normal subgroup of a given group $G$
with constant generating number on open subgroups.  The induction
base, when $G$ itself is saturable, was established using Lie ring
methods.  For $p \geq 3$ we can use the results in the present paper
to give a new proof of the base step as follows.  Suppose that $G$ is
saturable and $d(H)=d(G)$ for all open subgroups $H \leq G$.  Then
$d(H) = \dim(H)$ for all open subgroups $H \leq G$ and, by
Corollary~\ref{cor:sat-un}, the group $G$ is hereditarily uniform.
Hence $G$ is one of the groups described in
Corollary~\ref{cor:hered-uniform}.  This establishes
\cite[Corollary~2.4]{KlSn11}; for the induction step one proceeds in
the same way as in~\cite{KlSn11}.

\bigskip

\ackn The first author gratefully acknowledges support through a grant
of the London Mathematical Society as well as the support and
hospitality of the Instituto de Matem\'atica of the Universidade
Federal do Rio de Janeiro in~$2012$.

\bibliographystyle{plain}

\end{document}